\documentclass[12pt,oneside]{article}
\usepackage{amsmath,amssymb,amsfonts,amsthm}
\usepackage{color}
\usepackage[all]{xy}
\usepackage{graphicx}
\usepackage{color}
\usepackage{makeidx}
\usepackage{multirow}
\usepackage{rotating}

\newcommand{\T}{{\cal T}}

\newcommand{\set}[1]{\left\{#1\right\}}

\newcommand {\cp}{\mathfrak{X}(\pi (M))}

\newcommand {\N}{\mathcal{N}}

\def\Section#1{\vspace{30truept}\addtocounter{section}{1}\setcounter{thm}{0}\setcounter{equation}{0}
{\noindent\Large\bf\arabic{section}.#1}\par \vspace{12pt}}

\newtheorem{thm}{Theorem}[section]
\newtheorem{cor}[thm]{Corollary}
\newtheorem{lem}[thm]{Lemma}

\newtheorem{defn}[thm]{Definition}

\newtheorem{rem}[thm]{Remark}

\numberwithin{equation}{section}

\begin{document}
\title{\bf{A note on \lq\lq Sur le noyau de l'op\'{e}rateur de courbure d'une vari\'{e}t\'{e} finsl\'{e}rienne, C. R. Acad. Sci. Paris, s\'er. A, t. 272 (1971), 807-810}\rq\rq \footnote{ArXiv Number: 1305.4498 [math.DG]}}
\author{\bf{ Nabil L. Youssef$^{\,1,2}$ and S. G.
Elgendi$^{3}$}}
\date{}
\maketitle                     
\vspace{-1.16cm}
\begin{center}
{$^{1}$Department of Mathematics, Faculty of Science,\\ Cairo
University, Giza, Egypt}

\vspace{5pt}
$^{2}$Center for Theoretical Physics (CTP)\\
at the Britich University in Egypt (BUE)

\vspace{5pt}
{$^{3}$Department of Mathematics, Faculty of Science,\\ Benha
University, Benha,
 Egypt}
\end{center}

\begin{center}
E-mails: nlyoussef@sci.cu.edu.eg, nlyoussef2003@yahoo.fr\\
{\hspace{1.8cm}}salah.ali@fsci.bu.edu.eg, salahelgendi@yahoo.com
\end{center}
\smallskip
\vspace{1cm} \maketitle
\smallskip
{\vspace{-1.1cm}}
\noindent{\bf Abstract.}
In this note, adopting  the pullback formalism of global Finsler geometry,  we show by a\linebreak counterexample  that the kernel $\text{Ker}_R$ of the h-curvature $R$ of Cartan connection and the associated nullity distribution $\N_R$ do not coincide, contrary to  Akbar-Zadeh's result \cite{akbar.nul3.}.
We also give sufficient conditions for $\text{Ker}_R$ and $\N_R$ to coincide.

\medskip\noindent{\bf Keywords:\/}  Cartan connection, h-curvature tensor,  Nullity distribution, Kernel distribution.

\medskip\noindent{\bf  MSC 2010:\/} 53C60,
53B40, 58B20,  53C12.

\Section{ Introduction and notations }

\par

Nullity distribution in Finsler geometry has been investigated in \cite{akbar.nul3.} (adopting  the pullback formalism) and  \cite{Nabil.2} (adopting  the Klein-Grifone   formalism).
In 1971, Akbar-Zadeh \cite{akbar.nul3.}  proved that the kernel $\text{Ker}_R$ of the h-curvature operator $R$ of Cartan connection coincides with the nullity distribution $\N_R$ of that operator. This result was reappeared again in \cite{akbar.null.2} and was used to prove that the nullity foliation is auto-parallel. Moreover,  Bidabad and Refie-Rad \cite{bidabad} generalized this result to the case of k-nullity distribution following the same pattern of proof as Akbar-Zadeh's.
\par In this note, we show by a counterexample that $\text{Ker}_R$ and $\N_R$ do not coincide, contrary to Akbar-Zadeh's result. In addition, we find sufficient conditions for $\text{Ker}_R$ and $\N_R$ to coincide.

In what follows, we denote by $\pi: \T M\longrightarrow M$ the subbundle of nonzero vectors
tangent to $M$, $\pi_*: T(\T M)\longrightarrow TM$ the linear tangent map of $\pi$ and   $V_z(TM)=(\text{Ker}\,\pi_*)_z$ the vertical space at $z\in \T M$.
Let $\mathfrak{F}(TM)$ be the algebra of $C^\infty$ functions on $TM$ and $\cp$ the $\mathfrak{F}(TM)$-module of differentiable sections of the pullback bundle $\pi^{-1}(T M)$.
The elements of $\mathfrak{X}(\pi (M))$ will be called $\pi$-vector fields and denoted by barred letters $\overline{X}$.
The fundamental $\pi$-vector field is the $\pi$-vector field
$\overline{\eta}$ defined by $\overline{\eta}(z)=(z,z)$ for all
$z\in \T M$.

Let $D$ be  a linear connection  on the pullback bundle $\pi^{-1}(TM)$.
 Let $K$ be the map defined by
\vspace{-0.1cm} $K:T (\T M)\longrightarrow\pi^{-1}(TM):X\longmapsto D_X \overline{\eta}$.
The vector space $H_z (T M):= \{ X \in T_z(\T M) : K(X)=0 \}$ is the horizontal space to $M$ at $z$. The restriction of $\pi_*$ on $H_z (T M)$, denoted again $\pi_*$, defines an isomorphism between $H_z(TM)$ and $T_{\pi z}M$.
   The connection $D$ is said to be regular if
$ T_z (\T M)=V_z (T M)\oplus H_z (T M) \,\,  \forall \, z\in \T M$. In this case $K$ defines an isomorphism between $V_z(TM)$ and $T_{\pi z}M$.

 If $M$ is endowed with a regular connection, then the preceding decomposition permits to write uniquely a vector $X\in T_z(\T M)$ in the form $X=hX+vX$, where $hX\in H_z(TM)$ and $vX\in V_z(TM)$.
 The ((h)hv-) torsion tensor of $D$, denoted by $ T$, is  defined by
$ T(\overline{X},\overline{Y})=\textbf{T}(v{X},h {Y}), \text{for all}\,\,
\overline{X},\overline{Y}\in\mathfrak{X} (\pi (M)),$
where $\textbf{T}(X,Y)=D_X\overline{Y}-D_Y\overline{X}-\pi_*[X,Y]$ is the (classical) torsion  associated with $D$ and  $\overline{X}=\pi_* X$ (the fibers of the pullback bundle are isomorphic to the fibers of the tangent bundle).
The  h-curvature tensor
of $D$, denoted by $R$, is defined by
$R(\overline{X},\overline{Y})\overline{Z}=\textbf{K}(h
 {X},h  {Y})\overline{Z}, $
 where $\textbf{K}(X,Y)\overline{Z}=D_XD_Y\overline{Z}-D_YD_X\overline{Z}-D_{[X,Y]}\overline{Z}$
is the (classical) curvature   associated with $D$. The  contracted  curvature    $\widehat{R}$ is defined by
$\widehat{R}(\overline{X},\overline{Y})={R}(\overline{X},\overline{Y})\overline{\eta}.$


\Section{ Kernel and nullity distributions: Counterexample }

Let $(M,F)$ be a Finsler manifold. Let $\nabla$ be the Cartan connection associated with $(M,F)$.
It is well known that $\nabla$ is the unique metrical regular connection  on $\pi^{-1}(TM)$ such
that $g(T(\overline{X},\overline{Y}), \overline{Z})=g(T(\overline{X},\overline{Z}),\overline{Y})$ \cite{akbar.null.2}, \cite{r94}.
Note that the bracket $[X,Y]$ is horizontal if and only if $\widehat{R}(\overline{X},\overline{Y})=~0$, where $\widehat{R}$ is the contracted curvature of the $h$-curvature tensor of $\nabla$.

\begin{lem}\cite{akbar.null.2} \label{lem} Let $\textbf{T}$ and $\textbf{K}$ be the \emph{(}classical\emph{)} torsion and curvature tensors of $\nabla$  respectively.  We have:
$$\mathfrak{S}_{X,Y,Z}\{\textbf{K}(X,Y) \overline{Z}-\nabla_Z\textbf{T}(X,Y)-\textbf{T}(X,[Y,Z])\}=0,$$
where the symbol $\mathfrak{S}_{X,Y,Z}$ denotes cyclic sum over
$X, Y, Z\in \mathfrak{X}(TM)$.
\end{lem}

Let us now define the concepts of nullity and kernel spaces associated with the curvature $\textbf{K}$ of $\nabla$, following Akbar-Zadeh's definitions \cite{akbar.nul3.}.

\begin{defn}\label{nr}
The  subspace $\mathcal{N}_\bold{K}(z)$  of $H_z(TM)$ at a point $z\in TM$ is defined by
$$\mathcal{N}_\bold{K}(z):=\{X\in H_z(TM) : \,  \bold{K}(X,Y)=0, \, \,\forall\, Y\in H_z(TM)\}.$$
The dimension of $\mathcal{N}_\bold{K}(z)$ is denoted by $\mu_\bold{K}(z)$.

The subspace  $\mathcal{N}_\bold{K}(x):=\pi_*(\mathcal{N}_\bold{K}(z))\subset T_xM$, $x=\pi z$, is linearly isomorphic to $\mathcal{N}_\bold{K}(z)$. This subspace is called the nullity space of the curvature operator $\textbf{K}$ at the point $x\in M$
 \end{defn}
\begin{defn}\label{ker}
The kernel of ${\bold{K}}$ at the point $x=\pi z$ is defined by
$$\emph{\text{Ker}}_{\bold{K}}(x):=\{\overline{X}\in \{z\}\times T_xM\simeq T_xM:  \, {\bold{K}}(Y,Z) \overline{X}=0, \, \forall\, Y,Z\in H_z(TM)\}. $$
\end{defn}
Since $\N_\bold{K}$ and  $\text{Ker}_\bold{K}$  are both defined  on the horizontal space,  we can replace the classical curvature $\textbf{K}$ by  the h-curvature tensor $R$ of Cartan connection.
Akbar-Zadeh \cite{akbar.nul3.} proved that
the nullity space $\N_\bold{K}(x)$ and  the kernel space $\text{Ker}_\bold{K}(x)$ coincide for each point $x\in M $ at which they are defined.
We show by a counterexample  that the above mentioned spaces do not coincide.
\begin{thm}
The nullity space $\N_{R}(x)$ and  the kernel space $\emph{\text{Ker}}_R(x)$ do not coincide.
\end{thm}
 Let $M=\mathbb{R}^3$, $U=\{(x_1,x_2,x_3;y_1,y_2,y_3)\in \mathbb{R}^3 \times \mathbb{R}^3: x_3y_1>0,\,\, y_2^2+y_3^2\neq 0\}\subset TM$.  Let  $F$ be the Finsler function defined on $U$ by
 $$F:=\sqrt{{ x_3}{ y_1} \sqrt{{{ y_2}}^{2}+{{ y_3}}^{2}}}.$$

Using MAPLE program, we can perform the following computations. We write only  the coefficients $\Gamma^i_j$ of Barthel connection and the components $R^h_{ijk}$ of the h-curvature tensor $R$.

\noindent The non-vanishing coefficients of  Barthel connection $\Gamma^i_j$ are:
$$\Gamma^2_2=\frac{y_3}{x_3},\quad\quad \Gamma^2_3=\frac{y_2}{x_3},\quad\quad \Gamma^3_2=-\frac{y_2}{x_3},\quad\quad \Gamma^3_3=\frac{y_3}{x_3}.$$
The independent non-vanishing  components  of the h-curvature $R^h_{ijk}$ of   Cartan  connection  are:
$$R^1_{223}=\frac{y_1y_3}{2x_3^2(y_2^2+y_3^2)},\quad\quad R^1_{323}=-\frac{y_1y_2}{2x_3^2(y_2^2+y_3^2)},\quad\quad R^2_{123}=-\frac{y_3}{2x_3^2y_1},$$
{\vspace{-10pt}}$$R^2_{323}=-\frac{1}{2x_3^2},\quad\quad R^3_{123}=\frac{y_2}{2x_3^2y_1},\quad\quad R^3_{223}=\frac{1}{2x_3^2}.$$
Now, let $X\in \N_{{R}}$, then $X$ can be written in the form $X=X^1h_1+X^2h_2+X^3h_3$, where $X^1, X^2, X^3$  are the  components of the  vector $X$ with respect to the basis   $\{h_1, h_2, h_3\}$  of the horizontal space;  $h_i:=\frac{\partial}{\partial x^i}-\Gamma^m_i\frac{\partial}{\partial y^m}$, $i,m=1,...,3$.  The equation  ${R}(\overline{X},\overline{Y}) \overline{Z}=0$, $\forall\, Y,Z\in H(TM)$, is written locally in the form
  $X^j{R}^h_{ijk}=0.$
This is equivalent to the system of equations $X^2=0, \, X^3=0$ having the solution  $X^1=t \,(t\in \mathbb{R}), \,X^2=X^3=0$.
As $\pi_*(h_i)=\frac{\partial}{\partial x^i}$, we have
\begin{equation}\label{nullvector}
\N_R(x)=\set{t\frac{\partial}{\partial x^1}\,|\, t\in \mathbb{R}}.
\end{equation}
On the other hand, let $Z\in\text{Ker}_R$. The equation  ${R}(\overline{X},\overline{Y}) \overline{Z}=0$, $\forall\, X,Y\in H(TM)$, is written locally in the form
  $Z^i{R}^h_{ijk}=0.$
This is equivalent  to the system:
 $$y_3Z^2-y_2Z^3=0, \quad\quad y_3Z^1+y_1Z^3=0,\quad\quad y_2Z^1+y_1Z^2=0.$$
This system  has the solution $Z^1=t$,  $Z^2=-\frac{y_2}{y_1}t$ and $Z^3=-\frac{y_3}{y_1}t$, $(t\in \mathbb{R})$. Thus,
\begin{equation}\label{kervector}
\text{Ker}_R(x)=\set{t\Big(\frac{\partial}{\partial x^1}-\frac{y_2}{y_1}\frac{\partial}{\partial x^2}-\frac{y_3}{y_1}\frac{\partial}{\partial x^3}\Big)\,|\, t\in \mathbb{R}}.
\end{equation}
Comparing (\ref{nullvector}) and (\ref{kervector}), we note that there is no value of $t$ for which $\N_R(x)=\text{Ker}_R(x)$. Consequently,
$\N_R(x)$ and $ \text{Ker}_R(x)$ can not coincide. \qed

\vspace{8pt}
According to Akabr-Zadeh's proof, if $X\in \N_R$, then, by Lemma \ref{lem}, we have
$R(\overline{Y},\overline{Z}) \overline{X}={\textbf{T}}(X,[Y,Z]).$
But there is no guarantee for the vanishing of the right-hand side. Even the equation $g(R(\overline{Y},\overline{Z}) \pi_*{X},\pi_* W)=g({\textbf{T}}(X,[Y,Z]),\pi_* W)$, $W\in H(TM)$, is true only for $X\in \N_R$ and, consequently, we can not use the symmetry or skew-symmetry  properties in $X$ and $W$ to conclude that $g(R(\overline{Y},\overline{Z}) \overline{X},\overline{W})=0$. This can be assured, again,  by the previous example: if we take $X=h_1\in \N_R(z)$ and $Y=h_2,Z=h_3$, then the bracket $[Y,Z]=-\frac{y_3}{x_3^2}\frac{\partial}{\partial y_2}+\frac{y_2}{x_3^2}\frac{\partial}{\partial y_3}$ is vertical and $\textbf{T}(h_1,[h_2,h_3])=-\frac{1}{2x_3^2y_1}( y_3\bar{\partial}_2-y_2\bar{\partial}_3 )\neq 0$, where $\bar{\partial}_i$ is the basis of the fibers of the pullback  bundle.

\vspace{8pt}
As has been shown above, $\N_R$ and  $ \text{Ker}_R$ do not coincide in general. Nevertheless, we have

\begin{thm}Let $(M,F)$ be a Finsler manifold and $R$ the $h$-curvatire of Cartan connection. If
\begin{equation}\label{cyclic}
    \mathfrak{S}_{\overline{X},\overline{Y},\overline{Z}}R(\overline{X},\overline{Y})\overline{Z}=0,
 \end{equation}
then the two distributions $ \N_R$ and  $\text{Ker}_R$ coincide.
\end{thm}
\begin{proof}
If $X\in \N_R$, then, from  (\ref{cyclic}), we have $R(Y,Z)X=0$ and consequently $X\in \text{Ker}_R$. On the other hand, it follows also from (\ref{cyclic}) that $g(R(\overline{X},\overline{Y})\overline{Z},\overline{W})=:R(\overline{X},\overline{Y},\overline{Z},\overline{W})=
R(\overline{Z},\overline{W},\overline{X},\overline{Y})$. This proves that if $X\in \text{Ker}_R$, then $X\in \N_R$. 
\end{proof}

The following corollary shows that there are nontrivial cases in which (\ref{cyclic}) is verified and consequently the two distributions coincide.

\begin{cor}Let $(M,F)$ be a Finsler manifold and  $g$ the associated Finsler metric.\linebreak
If one of the following conditions holds:
\begin{description}
  \item[(a)] ${\widehat{R}}=0$ (the integrability condition for the horizontal distribution),
  \item[(b)]  $\widehat{R}(\overline{X},\overline{Y})=\lambda F(\ell(\overline{X})\overline{Y}-\ell(\overline{Y})\overline{X})$, where $\lambda(x,y)$ is  a homogenous function of degree $0$ in $y$ and  $\ell(\overline{X}):=F^{-1}g(\overline{X},\overline{\eta})$ (the isotropy condition),
\end{description}
then the two distributions $ \N_R$ and  $\text{Ker}_R$ coincide.
\end{cor}
\begin{proof}~\par
 \noindent \textbf{(a)} We have\,     $\mathfrak{S}_{\overline{X},\overline{Y},\overline{Z}}\{R(\overline{X},\overline{Y})\overline{Z}-T(\overline{X},\widehat{R}(\overline{Y},\overline{Z}))\}=0$ \cite{r95}. Then, if ${\widehat{R}}=0$, (\ref{cyclic}) holds.

\noindent \textbf{(b)} If $\widehat{R}(\overline{X},\overline{Y})=\lambda F(\ell(\overline{X})\overline{Y}-\ell(\overline{Y})\overline{X})$, then, by \cite{Soleiman}, (\ref{cyclic}) is satisfied.
\end{proof}

\begin{rem}
\em It should be noted that the identity (\ref{cyclic}) is a sufficient condition for the validity of the identity (2.1) of \cite{akbar.nul3.}.
\end{rem}

\providecommand{\bysame}{\leavevmode\hbox
to3em{\hrulefill}\thinspace}
\providecommand{\MR}{\relax\ifhmode\unskip\space\fi MR }
\providecommand{\MRhref}[2]{%
  \href{http://www.ams.org/mathscinet-getitem?mr=#1}{#2}
} \providecommand{\href}[2]{#2}

\end{document}